\documentclass[12pt,reqno]{article}

\usepackage[usenames]{color}
\usepackage{amssymb}
\usepackage{amsmath}
\usepackage{amsthm}
\usepackage{amsfonts}
\usepackage{amscd}
\usepackage{graphicx}
\usepackage{mathrsfs}

\usepackage[colorlinks=true,
linkcolor=webgreen,
filecolor=webbrown,
citecolor=webgreen]{hyperref}

\definecolor{webgreen}{rgb}{0,.5,0}
\definecolor{webbrown}{rgb}{.6,0,0}

\usepackage{color}
\usepackage{fullpage}
\usepackage{float}

\def\modd#1 #2{#1\ \mbox{\rm (mod}\ #2\mbox{\rm )}}

\begin{document}
\theoremstyle{plain}
\newtheorem{theorem}{Theorem}
\newtheorem{corollary}[theorem]{Corollary}
\newtheorem{lemma}[theorem]{Lemma}
\newtheorem{proposition}[theorem]{Proposition}

\theoremstyle{definition}
\newtheorem{definition}[theorem]{Definition}
\newtheorem{example}[theorem]{Example}
\newtheorem{conjecture}[theorem]{Conjecture}

\theoremstyle{remark}
\newtheorem{remark}[theorem]{Remark}
\newcommand{\h}{\hspace{2mm}}

\begin{center}
\vskip 1cm{\LARGE\bf 
Exact Frequencies of Consecutive\\ \vskip .1in
 Quadratic Residue and Nonresidue Patterns Modulo a Prime
  
}
\vskip 1cm
\large
Brahim Mittou \\
Laboratory of Applied Mathematics\\
University of Ouargla, Algeria \\
\href{mailto:mathmittou@gmail.com}{mathmittou@gmail.com}
\end{center}

\vskip .2 in

\begin{abstract}
This paper determines the exact frequency of all consecutive sign patterns of lengths two and three formed by the quadratic character modulo an odd prime $p$. Using basic properties of character sums over finite fields, we derive exact counting formulas for pairs $(n, n+1)$ and triples $(n-1, n, n+1)$ exhibiting specific sequence patterns in \(\{\pm 1\}\). The frequencies of pairs are classified entirely by \(p \pmod 4\), whereas triples exhibit a more complex dependency on \(p \pmod 8\) and the Jacobsthal sum $T_p=\sum_{x\in\mathbb F_p}\chi(x^3-x)$.
\end{abstract}



\section{Introduction} \label{s1}

Throughout this paper, for a given odd prime $p$, we let $\mathbb F_p$ denote the finite field $\mathbb{Z}/p\mathbb{Z}$ and we let $\chi$ denote the quadratic character modulo $p$.

Since the pioneering work of Gauss, quadratic residues and nonresidues modulo a prime have occupied a prominent place in number theory. One of the fundamental questions in this area concerns the local behavior of the sequence
$
\{\chi(n)\}_{n\in\mathbb F_p}
$.
A natural way to study this sequence is to investigate the occurrence of prescribed sign patterns. In particular, one may ask how frequently pairs or longer blocks of consecutive residues and nonresidues occur.

The study of such patterns dates back several decades. One of the earliest quantitative results is due to Carlitz \cite{2}, who obtained estimates for the number of pairs of consecutive quadratic residues and consecutive quadratic nonresidues in the finite field $\mathbb{F}_p$. A fundamental contribution of Burgess \cite{1} showed that, for sufficiently large primes $p$, the length of the longest run of consecutive quadratic residues or consecutive quadratic nonresidues modulo $p$ is bounded by $O(p^{1/4+\delta})$, where $\delta>0$ is arbitrary. Subsequent refinements and extensions of Burgess's method have led to stronger results; see, for example, \cite{5,7}. Using elementary techniques, Hummel \cite{3} proved the remarkable result that $13$ is the only prime for which the maximal length of a block of consecutive quadratic nonresidues exceeds $\sqrt{p}$.

More intricate residue patterns have also attracted considerable attention. Wang and Lv \cite{8} investigated the number of integers $1\le n\le p-1$ for which $n$, $n+\overline{n}$, and $n-\overline{n}$ are simultaneously quadratic residues or simultaneously quadratic nonresidues modulo $p$, where $\overline{n}$ denotes the multiplicative inverse of $n$ modulo $p$. More recently, Wang and Li \cite{9} examined the distribution of triples of consecutive quadratic residues and consecutive quadratic nonresidues modulo $p$, obtaining exact formulas for their frequencies.

While many existing results concern asymptotic estimates or the existence of long runs of residues and nonresidues, explicit formulas for the exact frequency of short sign patterns are less common. The purpose of this paper is to obtain complete and explicit formulas for all sign patterns of lengths two and three. For $(\varepsilon_1,\varepsilon_2)\in\{\pm1\}^2$,
we consider the quantity
$$
N(\varepsilon_1,\varepsilon_2)
=
\#\{n\in\mathbb F_p:
\chi(n)=\varepsilon_1,
\chi(n+1)=\varepsilon_2\},
$$
which counts the occurrences of a prescribed pair of consecutive quadratic character values. Our first theorem gives an exact evaluation of $N(\varepsilon_1,\varepsilon_2)$ and shows that the distribution depends only on the value of $\chi(-1)$, or equivalently on the residue class of $p$ modulo $4$.

We then study patterns of length three. For $(\varepsilon_1,\varepsilon_2,\varepsilon_3)\in\{\pm1\}^3$,
we define
$$
N(\varepsilon_1,\varepsilon_2,\varepsilon_3)
=
\#\{
n\in\mathbb F_p:
\chi(n-1)=\varepsilon_1,
\chi(n)=\varepsilon_2,
\chi(n+1)=\varepsilon_3
\}.
$$
Our second theorem expresses these quantities in terms of $
T_p=\sum_{x\in\mathbb F_p}\chi(x^3-x)
$ the Jacobsthal sum.
The resulting formulas depend explicitly on the residue class of $p$ modulo $8$, reflecting the influence of both $\chi(-1)$ and $\chi(2)$. This yields a complete classification of all triple patterns of consecutive quadratic residues and nonresidues.

\section{Preliminary Lemmas}

In order to prove our main results, the following lemmas will be useful.

\begin{lemma}\label{le1}
The values of $\chi$ at $-1$ and $2$ are given by:
\begin{equation}\label{eq1}
\chi(-1)=
(-1)^{\frac{p-1}{2}}=
\begin{cases}
1, &\text{if } p\equiv 1 \pmod 4,\\
-1, &\text{if } p\equiv 3 \pmod 4,
\end{cases}
\end{equation}
and
\begin{equation}\label{eq2}
\chi(2)=
(-1)^{\frac{p^2-1}{8}}=
\begin{cases}
1, &\text{if } p\equiv 1,7 \pmod 8,\\
-1, &\text{if } p\equiv 3,5 \pmod 8.
\end{cases}
\end{equation}
\end{lemma}
\begin{proof}
See, e.g., \cite[Chapter 5, Theorem 1]{4}.
\end{proof}

\begin{lemma}\label{le2}
If $f(x) \in \mathbb{F}_p[x]$ is a monic square-free polynomial, then 
\[
\sum_{x \in \mathbb{F}_p} \chi(f(x)) = 
\begin{cases}
0, & \text{if } \deg(f) = 1, \\
-1, & \text{if } \deg(f) = 2.
\end{cases}
\]
\end{lemma}
\begin{proof}
See, e.g., \cite[Theorems 1.10, 2.4]{6}.
\end{proof}

\begin{lemma}\label{le3}
Let $f(x)\in\mathbb F_p[x]$ be an odd polynomial. If $p\equiv 3\pmod 4$, then
\[
\sum_{x\in\mathbb F_p}\chi(f(x))=0.
\]
\end{lemma}

\begin{proof}
Since $p\equiv 3\pmod 4$, we have $\chi(-1)=-1$. Let
\[
S=\sum_{x\in\mathbb F_p}\chi(f(x)).
\]
Because $x\mapsto -x$ is a permutation of $\mathbb F_p$ and $f$ is odd,
\[
S=\sum_{x\in\mathbb F_p}\chi(f(-x))
=
\sum_{x\in\mathbb F_p}\chi(-1)\chi(f(x))
=
-\sum_{x\in\mathbb F_p}\chi(f(x))=-S.
\]
Hence, $2S=0$. Since $p$ is odd, the characteristic of $\mathbb F_p$ is not $2$, and it follows that $S=0$.
\end{proof}

\section{Main Results}

Our first result gives an explicit evaluation of the number of
consecutive pairs $(n,n+1)$ whose quadratic character values are
prescribed.
\begin{theorem}\label{th1}
For $(\varepsilon_1,\varepsilon_2)\in\{\pm1\}^2$ define
\[
N(\varepsilon_1,\varepsilon_2)
=
\#\{n\in\mathbb F_p:
\chi(n)=\varepsilon_1,
\chi(n+1)=\varepsilon_2\}.
\]
Then
\begin{equation}\label{eq3}
N(\varepsilon_1,\varepsilon_2)
=
\frac14
\begin{cases}
p-2-(\varepsilon_2
+\varepsilon_1
+\varepsilon_1\varepsilon_2),&  \text{ if }p\equiv1\pmod4,\\[2mm]
p-2-(\varepsilon_2-\varepsilon_1
+\varepsilon_1\varepsilon_2),&  \text{ if }p\equiv3\pmod4.
\end{cases}
\end{equation}
In particular, the distribution of pairs is symmetric,
\[
N(\varepsilon_1,\varepsilon_2)=N(\varepsilon_2,\varepsilon_1),
\]
if and only if \(p\equiv1\pmod4\).
\end{theorem}

\begin{proof}
Since $\varepsilon_1,\varepsilon_2\in\{\pm1\}$ and
$\chi(n)\in\{\pm1\}$ whenever $n\neq0$, the quantity
\[
\frac{1+\varepsilon\chi(n)}{2}
\]
is the indicator function of the condition $\chi(n)=\varepsilon$.
Indeed,
\[
\frac{1+\varepsilon\chi(n)}{2}
=
\begin{cases}
1,& \chi(n)=\varepsilon,\\[2mm]
0,& \chi(n)=-\varepsilon.
\end{cases}
\]
On the other hand, by definition, the target count $N(\varepsilon_1, \varepsilon_2)$ isolates elements where neither $n \equiv 0$ nor $n+1 \equiv 0 \pmod p$. Since $\chi(n), \chi(n+1) \in \{\pm 1\}$ whenever $n \notin \{0, -1\}$, the product function acts as a clean binary indicator. Therefore, we can write
\[
N(\varepsilon_1,\varepsilon_2)
= \frac14 \sum_{n \in \mathbb{F}_p \setminus \{0,-1\}} (1+\varepsilon_1\chi(n))(1+\varepsilon_2\chi(n+1)).
\]
Expanding the summand, we find
\begin{align*}
4N(\varepsilon_1,\varepsilon_2)
&=
\sum_{n\in\mathbb F_p\setminus\{0,-1\}}1
+\varepsilon_1
\sum_{n\in\mathbb F_p\setminus\{0,-1\}}\chi(n) \\
&\quad
+\varepsilon_2
\sum_{n\in\mathbb F_p\setminus\{0,-1\}}\chi(n+1)
+\varepsilon_1\varepsilon_2
\sum_{n\in\mathbb F_p\setminus\{0,-1\}}
\chi(n(n+1)).
\end{align*}
We now evaluate the sums. It follows from Lemma \ref{le2} that
\[
\sum_{n\in\mathbb F_p}\chi(n)=\sum_{n\in\mathbb F_p}\chi(n+1)=0,
\]
from which
\[
\sum_{n\in\mathbb F_p\setminus\{0,-1\}}\chi(n)
=
-\chi(-1),
\]
and similarly
\[
\sum_{n\in\mathbb F_p\setminus\{0,-1\}}\chi(n+1)
=
-\chi(1)
=-1.
\]
Moreover, by Lemma \ref{le2},
\[
\sum_{n\in\mathbb F_p\setminus\{0,-1\}}
\chi(n(n+1))
=
\sum_{n\in\mathbb F_p}
\chi(n(n+1))
=-1,
\]
since the terms corresponding to $n=0$ and $n=-1$ vanish.
Finally,
\[
\sum_{n\in\mathbb F_p\setminus\{0,-1\}}1=p-2.
\]
Substituting these evaluations gives
\[
N(\varepsilon_1,\varepsilon_2)
=
\frac14
\Bigl(
p-2
-\varepsilon_1\chi(-1)
-\varepsilon_2
-\varepsilon_1\varepsilon_2
\Bigr).
\]
Now using Identity \eqref{eq1}, we obtain the two cases in \eqref{eq3}.
This completes the proof.
\end{proof}

Now we generalize the counting problem to three
consecutive elements $(n-1,n,n+1)$ and obtain formulas that depend on
the residue class of $p$ modulo $8$ and on the Jacobsthal sum
$
T_p=\sum_{n\in\mathbb F_p}\chi(n^3-n)$.
\begin{theorem}\label{th2}
For
$(\varepsilon_1,\varepsilon_2,\varepsilon_3)\in\{\pm1\}^3$, define
\[
N(\varepsilon_1,\varepsilon_2,\varepsilon_3)
=
\#\Bigl\{
n\in\mathbb F_p:
\chi(n-1)=\varepsilon_1,\;
\chi(n)=\varepsilon_2,\;
\chi(n+1)=\varepsilon_3
\Bigr\}.
\]
\begin{enumerate}

\item If $p\equiv1\pmod8$, then
\begin{equation}\label{eq4}
N(\varepsilon_1,\varepsilon_2,\varepsilon_3)
=
\frac18
\Bigl(
p-3
-2\varepsilon_1
-2\varepsilon_2
-2\varepsilon_3
-2\varepsilon_1\varepsilon_2
-2\varepsilon_1\varepsilon_3
-2\varepsilon_2\varepsilon_3
+\varepsilon_1\varepsilon_2\varepsilon_3T_p
\Bigr).
\end{equation}

\item If $p\equiv3\pmod8$, then
\begin{equation}\label{eq5}
N(\varepsilon_1,\varepsilon_2,\varepsilon_3)
=
\frac18
\Bigl(
p-3
+\varepsilon_1\varepsilon_2\varepsilon_3T_p
\Bigr).
\end{equation}

\item If $p\equiv5\pmod8$, then
\begin{equation}\label{eq6}
N(\varepsilon_1,\varepsilon_2,\varepsilon_3)
=
\frac18
\Bigl(
p-3
-2\varepsilon_2
-2\varepsilon_1\varepsilon_3
+\varepsilon_1\varepsilon_2\varepsilon_3T_p
\Bigr).
\end{equation}

\item If $p\equiv7\pmod8$, then
\begin{equation}\label{eq7}
N(\varepsilon_1,\varepsilon_2,\varepsilon_3)
=
\frac18
\Bigl(
p-3
+2\varepsilon_1
-2\varepsilon_3
-2\varepsilon_1\varepsilon_2
-2\varepsilon_2\varepsilon_3
+\varepsilon_1\varepsilon_2\varepsilon_3T_p
\Bigr).
\end{equation}
\end{enumerate}
\end{theorem}

\begin{proof}
As in the proof of Theorem \ref{th1}, we can write
\[
N(\varepsilon_1,\varepsilon_2,\varepsilon_3)
=
\frac18
\sum_{n\in\mathbb F_p\setminus\{-1,0,1\}}
(1+\varepsilon_1\chi(n-1))
(1+\varepsilon_2\chi(n))
(1+\varepsilon_3\chi(n+1)).
\]
Expanding the product gives
\[
8N(\varepsilon_1,\varepsilon_2,\varepsilon_3)
=
S_0
+\varepsilon_1S_1
+\varepsilon_2S_2
+\varepsilon_3S_3
+\varepsilon_1\varepsilon_2S_{12}
+\varepsilon_1\varepsilon_3S_{13}
+\varepsilon_2\varepsilon_3S_{23}
+\varepsilon_1\varepsilon_2\varepsilon_3S_{123},
\]
where
\[
S_0=\sum_{n\in\mathbb F_p\setminus\{-1,0,1\}}1,\quad S_1=\sum_{n\in\mathbb F_p\setminus\{-1,0,1\}}\chi(n-1),
\]
\[
S_2=\sum_{n\in\mathbb F_p\setminus\{-1,0,1\}}\chi(n),\quad
S_3=\sum_{n\in\mathbb F_p\setminus\{-1,0,1\}}\chi(n+1),
\]
\[
\ S_{12}=\sum_{n\in\mathbb F_p\setminus\{-1,0,1\}}\chi(n^2-n),\quad S_{13}=\sum_{n\in\mathbb F_p\setminus\{-1,0,1\}}\chi(n^2-1),
\]
\[
S_{23}=\sum_{n\in\mathbb F_p\setminus\{-1,0,1\}}\chi(n^2+n),\quad S_{123}=\sum_{n\in\mathbb F_p\setminus\{-1,0,1\}}\chi(n^3-n).
\]
Clearly we have $S_0=p-3$. Since $\sum_{n\in\mathbb F_p}\chi(n)=0$, we obtain
\[
S_1
=
-\chi(-2)-\chi(-1)
=
-\chi(-1)(1+\chi(2)),
\]
\[
S_2
=
-\chi(-1)-1,
\]
and
\[
S_3
=
-\chi(2)-1.
\]
Next, Lemma \ref{le2} gives
\[
\sum_{n\in\mathbb F_p}\chi(n^2-n)=-1.
\]
Removing the contributions of $n=-1$, $0$, and $1$, we find
\[
S_{12}
=
-1-\chi(2).
\]
Similarly,
\[
S_{23}
=
-1-\chi(2).
\]
Also, by Lemma \ref{le2},
\[
\sum_{n\in\mathbb F_p}\chi(n^2-1)=-1.
\]
The contribution of the omitted values is
\[
0+\chi(-1)+0,
\]
from which
\[
S_{13}
=
-1-\chi(-1).
\]
Finally, since $n=-1,0,1$ are roots of $n^3-n$, their contributions are zero, and therefore
\[
S_{123}
=
\sum_{n\in\mathbb F_p}\chi(n^3-n)
=
T_p.
\]
Substituting the above evaluations gives
\[
N(\varepsilon_1,\varepsilon_2,\varepsilon_3)
=
\frac18
\Bigl(
p-3
-\varepsilon_1\chi(-1)(1+\chi(2))
-\varepsilon_2(1+\chi(-1))
-\varepsilon_3(1+\chi(2))
\]
\[
\qquad
-\varepsilon_1\varepsilon_2(1+\chi(2))
-\varepsilon_2\varepsilon_3(1+\chi(2))
-\varepsilon_1\varepsilon_3(1+\chi(-1))
+\varepsilon_1\varepsilon_2\varepsilon_3T_p
\Bigr).
\]
Formulas \eqref{eq4}--\eqref{eq7} now follow by inserting the values
of $\chi(-1)$ and $\chi(2)$ given by
\eqref{eq1} and \eqref{eq2}, according to the residue class of $p$
modulo $8$.
This completes the proof.
\end{proof}

\begin{corollary}
If $p\equiv3\pmod8$, then for every
$(\varepsilon_1,\varepsilon_2,\varepsilon_3)\in\{\pm1\}^3$,
\[
N(\varepsilon_1,\varepsilon_2,\varepsilon_3)
=
\frac{p-3}{8}.
\]
\end{corollary}

\begin{proof}
Since \(p\equiv3\pmod8\) and $f(x)=x^3-x$ is odd, Lemma \ref{le3} gives \(T_p=0\). Hence, \eqref{eq5} reduces to
\[
N(\varepsilon_1,\varepsilon_2,\varepsilon_3)
=
\frac{p-3}{8},
\]
which is independent of the prescribed sign pattern
$(\varepsilon_1,\varepsilon_2,\varepsilon_3)$.
\end{proof}


\begin{thebibliography}{99}
\bibitem{1} D. A. Burgess,
The distribution of quadratic residues and non-residues,
{\it Mathematika} {\bf 4} (1957), 106--112.

\bibitem{2}
L. Carlitz,
Sets of primitive roots,
{\it Compositio Mathematica} {\bf 13} (1956), 65--70.

\bibitem{3}
P. Hummel,
On consecutive quadratic non-residues: A conjecture of Issai Schur,
{\it J. Number Theory} {\bf 103} (2003), 257--266.

\bibitem{4}
K. Ireland and M. Rosen,
{\it A Classical Introduction to Modern Number Theory},
Springer, New York, 1990.

\bibitem{5}
Y. K. Lau and J. Wu,
On the least quadratic non-residues,
{\it International Journal of Number Theory}
{\bf 4} (2008), 423--435.

\bibitem{6}
B.~Nica, {\it Jacobsthal Sums}, Monographs in Number Theory, vol.~14, World Scientific Publishing Co., Singapore, 2025.

\bibitem{7}
A. Schinzel,
Primitive roots and quadratic non-residues,
{\it Acta Arithmetica}
{\bf 149} (2011), 161--170.

\bibitem{8}
T. T. Wang and X. X. Lv,
The quadratic residues and some of their new distribution properties,
{\it Symmetry}
{\bf 12} (2020), 421--428.

\bibitem{9}
X. Wang and A. H. Li,
Distribution properties and applications of consecutive quadratic residues,
{\it Acta Mathematica Sinica (Chinese Series)}
{\bf 66} (2023), 1--10.


\end{thebibliography}
\end{document}